\documentclass[a4paper,12pt]{amsart}

\usepackage{amsmath,amssymb,amsthm}
\usepackage[utf8]{inputenc}
\usepackage[english]{babel}
\usepackage[alphabetic]{amsrefs}

\usepackage{etoolbox}
\apptocmd{\sloppy}{\hbadness 10000\relax}{}{}

\newtheorem{theorem}{Theorem}[section]
\newtheorem{corollary}[theorem]{Corollary}

\newtheorem{lemma}[theorem]{Lemma}
\newtheorem{proposition}[theorem]{Proposition}
\theoremstyle{definition}
\newtheorem{definition}[theorem]{Definition}
\theoremstyle{remark}
\newtheorem{remark}[theorem]{Remark}
\newtheorem{example}[theorem]{Example}
\newtheorem*{questions*}{Questions}

\newcommand{\NN}{\mathbb{N}}

\newcommand{\RR}{\mathbb{R}}
\newcommand{\CC}{\mathbb{C}}
\DeclareMathOperator{\adj}{ad}

\newcommand{\aut}{\mathop{\mathrm{Aut}}}
\newcommand{\cont}{\mathcal{C}}

\DeclareMathOperator{\id}{id}

\title[Lie algebra of vector fields on Danielewski surfaces]{On complete generators of certain Lie algebras on Danielewski surfaces}
\author{Rafael B.\ Andrist}
\address{Faculty of Mathematics and Physics, 
University of Ljubljana, Ljubljana, Slovenia}
\email{rafael-benedikt.andrist@fmf.uni-lj.si}

\begin{document}

\begin{abstract}
We study the Lie algebra of polynomial vector fields on a smooth Danielewski surface of the form $x y = p(z)$ with $x,y,z \in \mathbb{C}$. We provide explicitly given generators to show that:
1. The Lie algebra of polynomial vector fields is generated by $6$ complete vector fields. 2. The Lie algebra of volume-preserving polynomial vector fields is generated by finitely many complete vector fields, whose number depends on the degree of the defining polynomial $p$. 3. There exists a Lie sub-algebra generated by $4$ LNDs whose flows generate a group that acts infinitely transitively on the Danielewski surface. The latter result is also generalized to higher dimensions where $z \in \mathbb{C}^N$.
\end{abstract}



\maketitle

\section{Introduction}

Danielewski surfaces have been introduced in the study of the cancellation problem in algebraic geometry, see e.g.\ the survey article by Kraft \cite{MR1423629}. Since then, they have become a well-studied object in algebraic geometry and complex analysis. In particular, the Lie algebra of polynomial vector fields on a Danielewski surface and both the algebraic and holomorphic automorphism group have been investigated from various aspects. 

The notion of infinite transitivity of a group action has first been used by Kaliman and Zaidenberg \cite{MR1669174}. The definition can easily by extended to singular spaces and the holomorphic category \cite{finiteliebis}:

\begin{definition}
Let $X$ be a complex variety and let $G$ be a group acting on $X$ through
(algebraic or holomorphic) automorphisms, then we call the action of $G$ \emph{infinitely transitive} if $G$ acts on the regular part $X_{\mathrm{reg}}$ $m$-transitively for any $m \in \NN$.
\end{definition}

Recently, the study of groups that act infinitely transitively, but are generated by finitely many unipotent groups, has been initiated by Arzhantsev, Kuyumzhiyan and Zaidenberg \cite{MR3949984} who showed that this property holds for toric varieties that are smooth in codimension $2$. Alternatively, one can also aim to generate the whole Lie algebra of polynomial vector fields with finitely many complete vector fields whose flows will then -- by an application of Anders\'en--Lempert theory -- also generate a group that acts infinitely transitively. This approach was pursued by the author for $\CC^n$ \cite{finitelie}, and for $\mathrm{SL}_2(\CC)$ and the singular surface $x y = z^2$ \cite{finiteliebis}.

In this article we will consider these questions for the Danielewski surface $Z_p := \{ (x,y,z) \in \CC^3 \,:\, x y - p(z) = 0 \}$
where $p(z) \in \CC[z]$ is a polynomial with simple zeros.
Requiring that $p$ has only simple zeros is equivalent to the smoothness of $Z_p$. We will denote its degree by $d := \deg p$.

Throughout the paper we will work with the following, well-known vector fields on $Z_p$:
\[
\begin{alignedat}{4}
V   &:= p'(z) \frac{\partial}{\partial x} & &+ y \frac{\partial}{\partial z} \\
W   &:= & p'(z) \frac{\partial}{\partial y} &+ x \frac{\partial}{\partial z} \\
H   &:= - x \frac{\partial}{\partial x} &+ y \frac{\partial}{\partial y} &
\end{alignedat}
\]
Note that $[V, H] = -V$, $[W, H] = W$ and $[V, W] = p''(z) \cdot H$. 
The flows of $V, W$ and $H$, respectively, are given by
\begin{align*}
\varphi_t(x,y,z) &= (x + t p'(z) + \dots + t^d y^{d-1} p^{(d)}(z), y, z + y t) \\ 
\psi_t(x,y,z)    &= (x, y + t p'(z) + \dots + t^d x^{d-1} p^{(d)}(z), z + x t) \\ 
\chi_t(x,y,z)    &= (e^{-t} x, e^{t} y, z)
\end{align*}

The smooth Danielewski surface $Z_p$ is equipped with a holomorphic volume form 
\[
\omega := \frac{dx \wedge dz}{x} = \frac{dz \wedge dy}{y} = \frac{dx \wedge dy}{p'(z)}.
\]
It is straightforward to check that the vector fields $V, W, H$ preserve the volume form $\omega$, i.e.\ the Lie derivative of $\omega$ w.r.t.\ to each of them vanishes: $\mathcal{L}_{\Theta} \omega = 0$ for $\Theta = V, W, H$.

\begin{definition}[\cite{MR1721579}]
Let $X$ be a complex manifold and let $\Theta$ be a vector field on $X$. 
For a holomorphic function $f \colon X \to \CC$ we call the vector field
\begin{enumerate}
\item $f \cdot \Theta$ a \emph{shear} vector field of $\Theta$ if $\Theta(f) = 0$.
\item $f \cdot \Theta$ an \emph{overshear} vector field of $\Theta$ if $\Theta^2(f) = 0$.
\end{enumerate}
\end{definition}
The significance of this definition lies in the following facts \cite{MR1721579}: If the vector field $\Theta$ is complete (i.e.\ its flow map exists for all complex times), then the shear and overshear vector fields are complete as well. Moreover, if $X$ admits a volume form, and $\Theta$ is volume-preserving, then also its shear vector fields are volume-preserving. In the algebraic category, a shear $f \cdot \Theta$ of a locally nilpotent derivation $\Theta$ is called a \emph{replica} of $\Theta$, and is again a locally nilpotent derivation.

\begin{example}
The vector field $V$ is a locally nilpotent derivation and hence a complete and $\omega$-preserving vector field on the Danielewski surface $Z_p$. Since $f(y) \in \CC[y]$ is in the kernel of $Y$, the vector field $f(y) V$ is a shear of $V$ and, hence, complete and $\omega$-preserving as well. In fact, the flow map of $f(y) V$ is given by $\varphi_{f(y) t}$. On the other hand, we have $\Theta^2(z f(y)) = \Theta(y f(y)) = 0$, which means that $z f(y) V$ is an overshear of $V$ and again complete; however, it does no longer preserve $\omega$: $\mathcal{L}_{z f(y) V} \omega = y f(y) \cdot \omega$. 
\end{example}

The smooth Danielewski surfaces $Z_p$ enjoy the so-called algebraic density property and the algebraic volume density property, see Kaliman and Kutzschebauch \cite{MR2350038}. The algebraic density property implies the density property, and the algebraic volume density property implies the volume density property. These properties allow for a Runge approximation of (volume-preserving) holomorphic injections by (volume-preserving) holomorphic automorphisms, and for a description of the (volume-preserving) holomorphic automorphism group. See Section \ref{sec-survey} for more details. The Lie algebra generated by the overshears of $V$ and $W$ was studied in further detail by Kutzschebauch and Lind \cite{MR2823038}, and the Lie algebra generated by all the locally nilpotent derivations was determined by Kutzschebauch and Leuenberger \cite{MR3495426}.

The following are the three main results of this paper, each of which determines a finite set of generators for a certain Lie algebra of vector fields on a smooth Danielewski surface.

\begin{enumerate}
\item The Lie algebra of all polynomial vector fields on $Z_p$ is generated by (at most) six complete vector fields, namely \[V, W, H, zV, zW, zH\] (see Theorem \ref{thm-full}).
\item The Lie algebra of volume-preserving polynomial vector fields on $Z_p$ is generated by finitely many complete vector fields, namely \[y^n V, x^n W, z^m H\] for  $0 \leq n \leq \max\{1 , d-2\}$ and $0 \leq m \leq \max\{2, d-2\}$ (see Theorem \ref{thm-volume}).
\item However, an infinitely transitive action on $Z_p$ is generated by four unipotent groups, namely the flows of the locally nilpotent derivations \[V, W, yV, xW\] (see Theorem \ref{thm-unipotent}).
\item This result can be extended to higher dimensions (see Theorem \ref{thm-affinemod}).
\item We also consider the case of direct products in Section \ref{sec-products} in a general setting. This is used in the proof of the preceding result.
\end{enumerate}

The vector fields $V, W, yV, xW$ do in general (for $d \geq 4$) generate a Lie algebra that contains not even all of the powers $y^n V, x^n W$. However, it turns out that despite these ``gaps'', we can still interpolate functions in their respective kernels sufficiently well (see Lemma \ref{lem-interpolate}) to construct the automorphisms needed for infinite transitivity. 

\section{Brief Survey of the Density Property}
\label{sec-survey}

This brief survey basically follows the exposition of the author in \cite{finiteliebis}.

\begin{definition}
Let $X$ be a complex manifold and let $\Theta$ be a holomorphic vector field on $X$. We call $\Theta$ \emph{complete} or \emph{$\CC$-complete} if its flow map exists for all times $t \in \CC$. We call $\Theta$ \emph{$\RR$-complete} if its flow map exists for all times $t \in \RR$. 
\end{definition}
Since the flow satisfies the semi-group property, any time-$t$ map of a $\RR$- or $\CC$-complete vector field is a holomorphic automorphism.

\smallskip

The density property for complex manifolds was introduced and studied by Varolin in \cites{MR1785520,MR1829353} around 2000:

\begin{definition} [\cite{MR1829353}] \hfill
\begin{enumerate}
\item Let $X$ be a Stein manifold. We say that $X$ has the \emph{density property} if the Lie algebra generated by the complete holomorphic vector fields on $X$ is dense (in the compact-open topology) is the Lie algebra of all holomorphic vector fields on $X$.
\item Let $X$ be an affine manifold. We say that $X$ has the \emph{algebraic density property}, if the Lie algebra generated by the complete algebraic vector fields on $X$ coincides with the Lie algebra of all algebraic vector fields on $X$.
\end{enumerate}
\end{definition}

By a standard application of Cartan's Theorem B and Cartan--Serre's Theorem A, the algebraic density property implies the density property. The algebraic density property is merely a tool to prove the density property, since algebraic vector fields are easier to describe.

\begin{example}
For the purpose of this article, we only mention the following examples
\begin{enumerate}
\item $\CC^n, n \geq 2,$ enjoys the algebraic density property
\item $(\CC^n, dz_1 \wedge \dots \wedge dz_n)$ enjoys the algebraic volume density property
\item $((\CC^\ast)^n, dz_1/z_1 \wedge \dots \wedge dz_n/z_n)$ enjoys the algebraic volume density property
\item Danielewski surfaces $\{ (x,y,z) \in \CC^3 \,:\, xy = p(z) \}$ where $p$ is a polynomial with simple zeroes, enjoy the algebraic density property, and, with volume form $\omega = dx/x \wedge dz$, the algebraic volume density property \cite{MR2350038}.
\end{enumerate}
For details and a comprehensive list we refer the reader to the recent survey by Forstneri\v{c} and Kutzschebauch \cite{MR4440754}. These proofs require countable families of complete vector fields, and it is a priori not clear that the Lie algebras in question can be finitely generated.
\end{example}

Let $X$ be a complex manifold of complex dimension $n$. We call a complex differential form of bi-degree $(n,0)$ on $X$ a \emph{volume form} if it is nowhere degenerate.

Let $X$ be a complex manifold. We denote its \emph{group of holomorphic automorphisms} by $\aut(X)$. If $X$ is smooth and if there exists a volume form $\omega$ on $X$, we denote the \emph{group of $\omega$-preserving holomorphic automorphisms} by $\aut_{\omega}(X)$.

\begin{definition}\hfill
\begin{enumerate}
\item Let $X$ be a Stein manifold with a holomorphic volume form $\omega$. We say that $(X, \omega)$ has the \emph{volume density property} if the Lie algebra generated by the complete $\omega$-preserving holomorphic vector fields on $X$ is dense (in the compact-open topology) in the Lie algebra of all $\omega$-preserving holomorphic vector fields on $X$. \cite{MR1829353}
\item Let $X$ be an affine manifold with an algebraic volume form $\omega$. We say that $(X, \omega)$ has the \emph{algebraic volume density property} if the Lie algebra generated by the complete $\omega$-preserving algebraic vector fields on $X$ coincides with the Lie algebra of all $\omega$-preserving algebraic vector fields on $X$. \cite{MR2660454}
\end{enumerate}
\end{definition}

Again, the algebraic volume density property implies the volume density property; however, the proof is not straightforward and can be found in \cite{MR2660454} by Kaliman and Kutzschebauch.

The main result for manifolds with the density property is the following theorem which was first stated for star-shaped domains of $\CC^n$ by Anders\'en and Lempert in 1992, then generalized to Runge domains by Forstneri\v{c} and Rosay in 1993 and finally extended to manifolds with the density property by Varolin:
\begin{theorem}\cite{MR1829353}
\label{thmAL}
Let $X$ be a Stein manifold with the density property or $(X,\omega)$ be a Stein manifold with the volume density property, respectively. Let $\Omega \subseteq X$ be an open subset and $\varphi \colon [0,1] \times \Omega \to X$ be a $\cont^1$-smooth map such that
\begin{enumerate}
\item $\varphi_0 \colon \Omega \to X$ is the natural embedding,
\item $\varphi_t \colon \Omega \to X$ is holomorphic and injective for every $t \in [0,1]$ and, respectively, $\omega$-preserving, and
\item $\varphi_t(\Omega)$ is a Runge subset of $X$ for every $t \in [0,1]$.
\end{enumerate}
In the case of volume density property, further assume that the holomorphic $(n-1)$th de Rham cohomology of $\Omega$ vanishes. 
Then for every $\varepsilon > 0$ and for every compact $K \subset \Omega$ there exists a continuous family $\Phi \colon [0, 1] \to \aut(X)$ or (respectively) $\Phi \colon [0, 1] \to \aut_\omega(X)$, such that $\Phi_0 = \id_X$ and $\| \varphi_t - \Phi_t \|_K < \varepsilon$ for every $t \in [0,1]$.

\smallskip
Moreover, each of these automorphisms $\Phi_t$ can be chosen to be compositions of flows of generators of a dense Lie subalgebra in the Lie algebra of all holomorphic vector fields on $X$.
\end{theorem}

One of the two main ingredients in the proof of Theorem \ref{thmAL} is the following proposition which has been found by Varolin \cite{MR1829353}, but is stated best as a stand-alone result in the monograph of Forstneri\v{c} \cite{MR3700709}.

\begin{proposition}[Forstneri\v{c} \cite{MR3700709}*{Corollary 4.8.4}]
\label{prop-lieapprox}
Let $V_1, \dots, V_m$ be $\mathbb{R}$-complete holomorphic vector fields on a complex manifold $X$. Denote by $\mathfrak{g}$ the Lie subalgebra generated by the vector fields $\{V_1, \dots, V_m\}$ and let $V \in \mathfrak{g}$. Assume that $K$ is a compact set in $X$ and $t_0 > 0$ is such that the flow $\varphi_t(x)$ of $V$ exists for every $x \in K$ and for all $t \in [0, t_0]$. Then $\varphi_{t_0}$ is a uniform limit on $K$ of a sequence of compositions of time-forward maps of the vector fields $V_1, \dots, V_m$.
\end{proposition}

For the proof of Theorem \ref{thm-unipotent}, where we can't make use of the density property, we will need to use Proposition \ref{prop-lieapprox} directly.

\bigskip

As one of many standard applications of Theorem \ref{thmAL} we obtain the following. It is implicit in the paper of Varolin \cite{MR1785520}, but can also be found with a detailed proof in \cite{finitelie}*{Lemma 7 and Corollary 8}.

\begin{proposition}
\label{propinftrans}
Let $X$ be a Stein manifold with the density property resp.\ $(X,\omega)$ be a Stein manifold with the volume density property with $\dim_\CC X \geq 2$. Let $\mathfrak{g}$ be a Lie algebra that is dense in the Lie algebra of all holomorphic vector fields on $X$ resp.\ in the Lie algebra of all $\omega$-preserving holomorphic vector fields on $X$. Then the group of holomorphic automorphisms generated by the flows of completely integrable generators of $\mathfrak{g}$ acts infinitely transitively on $X$.
\end{proposition}

\section{Lie algebra of polynomial vector fields}

\begin{lemma}
\label{lem-stepup}
The complete vector fields $V, W, z W, z V$ generate a Lie algebra containing $y^n V, x^n W$ for $n \geq 0$.
\end{lemma}
\begin{proof}
By induction we obtain $[x^n W, z W] = x^{n+1} W$ and $[y^n V, z V] = y^{n+1} W$. 
\end{proof}

\begin{theorem}
\label{thm-full}
The complete vector fields $V, W, H, z W, z V, z H$ generate the Lie algebra of all polynomial vector fields on the smooth Danielewski surface $Z_p$. 
\end{theorem}

\begin{proof}
We first obtain
\[
-y H = [zH, V] - [H, zV]
.\] 
Then we proceed by induction on $n \geq 0$ with
\begin{align*}
[y H, z y^n V] &= y H(z y^n) V - z y^n V(y) H + z y^{n+1} [H, V] \\
 &= (n + 1) z y^{n+1} V
\end{align*}
Similarly, we obtain the terms $z x^{n+1} W$. 
Together with Lemma \ref{lem-stepup} we therefore obtain the following shear and overshear vector fields:
\[
y^n V, x^n W, z y^n V, z x^n W, \quad n \geq 0
\]
It was shown by Kutzschebauch and Lind \cite{MR2823038} that these vector fields together generate the Lie algebra of all polynomial vector fields on a smooth Danielewski surface $Z_p$.
\end{proof}

\begin{remark}
Note that in case of $d = 1$, the Danielewski surface $Z_p$ is just a graph and thus algebraically isomorphic to $\CC^2$ where we know by a previous result of the author \cite{finitelie} that only three complete generators are needed, however of less symmetric shape. 
\end{remark}

\begin{corollary}
The group generated by the one-parameter groups corresponding to the vector fields $V, W, H, z W, z V, z H$ is dense in the identity component of the group of holomorphic automorphisms of $Z_p$. In particular, it acts infinitely transitively on $Z_p$.
\end{corollary}
\begin{proof}
Since $V, W, H, z W, z V, z H$ are complete and generate the Lie algebra of all polynomial vector fields, the result follows Theorem \ref{thmAL} and Proposition \ref{propinftrans}.
\end{proof}

\section{Infinite Transitivity}

We need several computational lemmas as a preparation. 

\label{sec-LND}

For the iterated action of $\adj_V = [V, \cdot]$ we obtain the following formula:
\begin{lemma}
\label{lem1}
\begin{align}
\adj_V^n(W)        &=  -(n-1) y^{n-2} p^{(n)} \cdot V + y^{n-1} p^{(n+1)}(z) \cdot H
\end{align}
\end{lemma}
\begin{proof}
We have that $\adj_V(W) = [V, W] = p''(z) \cdot H$ which proves the case $n=1$. We now proceed by induction:
\begin{align*}
\adj_V^{n+1}(W) &=  [V,-(n-1) y^{n-2} p^{(n)} \cdot V + y^{n-1} p^{(n+1)}(z) \cdot H] \\
 &= V(-(n-1) y^{n-2} p^{(n)}) \cdot V \\
 &\quad- y^{n-1} p^{(n+1)}(z) \cdot V + V(y^{n-1} p^{(n+1)}(z)) \cdot H \\
 &= -n y^{n-1} p^{(n+1)} \cdot V + y^{n} p^{(n+2)}(z) \cdot H \qedhere
 \end{align*}
\end{proof}

By Lemma \ref{lem1} we obtain the term $y^{d-2} V$ if we choose $n = d-1$ where $d = \deg p$. 

\begin{remark}
We will sometimes need the following identity of vector fields on the Danielewski surface:
\begin{equation}
-x V + y W = p'(z) H
\end{equation}
\end{remark}

\begin{lemma}
\label{lem3}
\begin{align*}
\adj^n_{yV}(W) &= -n \cdot y^{2n-2} \cdot p^{(n)} \cdot V + y^{2n-1} \cdot p^{(n+1)} \cdot H \\
\adj^m_{V}(\adj^n_{yV}(W)) &= -(n+m) \cdot y^{2n-2+m} \cdot p^{(n+m)} \cdot V \\ &\quad + y^{2n-1+m} \cdot p^{(n+1+m)} \cdot H 
\end{align*}
\end{lemma}

\begin{proof}
We first proceed by induction on $n$.
\begin{align*}
[yV, W] &= y [V, W] - W(y) V = y p''(z) H - p'(z) V \\
[yV, \adj^n_{yV}(W)] &= -n \cdot y^{2n} \cdot p^{(n+1)} \cdot V + y^{2n+1} \cdot p^{(n+2)} \cdot H \\ &\quad - y^{2n} \cdot p^{(n+1)} \cdot V 
\end{align*}
Next, we proceed by induction on $m$, and note that this formula also holds for $m=0$.
\begin{align*}
[V, \adj^m_{V}(\adj^n_{yV}(W))] &= -(n+m) \cdot y^{2n-1+m} \cdot p^{(n+m+1)} \cdot V \\ &\quad + y^{2n+m} \cdot p^{(n+2+m)} \cdot H \\ &\quad- \cdot y^{2n-1+m} \cdot p^{(n+m + 1)} \cdot V \qedhere
\end{align*}
\end{proof}

\begin{corollary}
The Lie algebra generated by $V, W, yV$ contains the vector fields
\[
y^{d-2} V, \; y^{d-1} V, \; y^{d} V, \; \dots, \; y^{2d-2} V
\]
\end{corollary}

\begin{lemma}
\label{lem4}
\begin{equation*}
\adj^n_V([y^{k} V, W]) = y^{k + n} \cdot p^{(n+2)} \cdot H - (k + n) \cdot y^{k - 1 + n} \cdot p^{(n + 1)} \cdot V 
\end{equation*}
\end{lemma}
\begin{proof}
\begin{align*}
[y^{k} V, W]) &= y^{k} \cdot p'' \cdot H - k \cdot y^{k - 1} \cdot p' \cdot V \\
[V, \adj^n_V([y^{k} V, W])] &= [V, y^{k + n} \cdot p^{(n+2)} \cdot H] \\
&\quad - [V,(k + n) \cdot y^{k - 1 + n} \cdot p^{(n + 1)} \cdot V] \\
&= y^{k + n + 1} \cdot p^{(n+3)} \cdot H - y^{k + n} \cdot p^{(n+2)}  \cdot V \\ 
&\quad - (k + n) \cdot y^{k + n} \cdot p^{(n + 2)} \cdot V \\
&= y^{k + n + 1} \cdot p^{(n+3)} \cdot H \\ &\quad - (k + n + 1) \cdot y^{k + n} \cdot p^{(n + 2)} \cdot V \qedhere
\end{align*}
\end{proof}

\begin{corollary}
\label{cor-unipotentpowers}
The Lie algebra generated by $V, W, yV$ contains the vector fields
\[
y^{n} V, \quad n \geq d-2
\]
\end{corollary}
\begin{proof}
We set $n = d - 1$ to obtain a vector field which is a multiple of $V$. The power of $y$ in the coefficient of $V$  is $d + k - 2$. By the previous calculations, we can assume to have already obtained the powers $k = d-2, \dots, 2d-2$, and hence now obtain 
$2d - 4, \dots, 3d - 4$. We then proceed by induction.
\end{proof}

\begin{remark}
All the vector fields involved in the calculations in this section preserve the volume form $\omega$. Therefore, one might argue that the calculations could be simplified by using the Poisson bracket, see Section \ref{sec-volume}. However, in Section \ref{sec-suspension} we will work in a more general situation where a volume form might not exist, and hence it is actually necessary to carry out the computations using Lie bracket of vector fields.
\end{remark}

\begin{lemma}
\label{lem-interpolate}
Let $a_1, \dots, a_m \in \CC^\ast$ be pairwise disjoint points. For any $r \in \NN$ there exists a polynomial $f(w) \in w^r \CC[w]$ 
such that $f(a_1) = \dots = f(a_{m-1}) = 0$ and $f(a_m) = 1$.
\end{lemma}
\begin{proof}
We can find $f$ in the following form:
\[
f(w) = c \cdot w^r \cdot (w - a_1) \cdots (w - a_{m-1})
\] 
where $c \in \CC$ is chosen such that $f(a_m) = 1$.
\end{proof}

\begin{lemma}
\label{lem-nonvanish}
The vector fields $V$ and $W$ never vanish on a smooth Danielewski surface. Moreover, they span the tangent space $T_{(x,y,z)} Z_p$ except in the points 
$(x,y,z)$ with $p'(z) = 0$ and $x y \neq 0$. 
\end{lemma}
\begin{proof}
The vector field $V$ could only vanish if $y = 0 \Longrightarrow p(z) = 0$ and $p'(z) = 0$ which is excluded since $p$ has only simple zeros. Similarly, $W$ can't vanish either.
If $V$ and $W$ are linearly dependent, then necessarily $p'(z) = 0 \Longrightarrow p(z) = xy \neq 0$.
\end{proof}

\begin{lemma}
\label{lem-generalpos}
Let $p_1, \dots, p_m, q \in Z_p$ be pair-wise different points. Then there exist flow times $t, s \in \CC$ such that
$p_1^\prime := \psi_s(\varphi_t(p_1)), \dots, \\ p_m^\prime := \psi_s(\varphi_t(p_m)), q^\prime :=\psi_s(\varphi_t(q)) \in Z_p$ are such that:
\begin{enumerate}
\item Their $z$-coordinate satisfies $p'(z) \neq 0$ and $p(z) \neq 0$.
\item Their $x$-coordinates are pair-wise different and non-zero.
\item Their $y$-coordinates are pair-wise different and non-zero.
\end{enumerate}
\end{lemma}
\begin{proof}
\[
\begin{split}
&\psi_s(\varphi_t(x,y,z)) = \\ &(x + t p'(z) + \dots, y + s \cdot p'(z + yt) + \dots, z + yt + sx + st p'(z) + \dots)
\end{split}
\]
For any points with $z$-coordinate such that $p'(z) = 0$, we have $p(z) \neq 0 \Longleftrightarrow xy \neq 0$. The zeros and critical points of $p$ are isolated points in $\CC$. Hence, we can read off from the formula above, that the first condition is satisfied for a dense open set of flow times $(t, s) \in \CC^2$. Assume now that two different points have the same $x$-coordinate. Since $0 \neq p(z) = xy$ determines $y$, their $z$-coordinates must differ. Flowing along $V$ changes their $x$-coordinates for an open and dense set of times $t \in \CC$.
Their $x$-coordinates will change according to $p(z + t y) / y$ and hence must be different for all but a finite number of exceptions in $t \in \CC$. 
Similarly, we treat the $y$-coordinates by flowing along $W$.
Since the finite intersection of open and dense sets is dense, we find the desired flow times $(t,s) \in \CC^2$.  
\end{proof}

\begin{theorem}
\label{thm-unipotent}
The group generated by the unipotent groups corresponding to the four locally nilpotent derivations $V, W, yV, xW$ acts infinitely transitively on the smooth Danielewski surface $x y = p(z)$.
\end{theorem}
\begin{proof}
Let $\mathfrak{g}$ denote the Lie algebra generated by $V, W, yV, xW$ and let $G$ denote the group generated by the flows of the LNDs $V, W, yV, xW$. 

The statement of the theorem follows directly from the following

\textbf{Claim:} For any given pair-wise different points $p_1, \dots, p_m, q$ in the Danielewski surface, there exists an automorphism $\alpha \in G$ such that 
\begin{equation}
\alpha(p_1) = p_1, \dots, \alpha(p_{m-1}) = p_{m-1}, \alpha(p_1) = q
\end{equation}
We now prove this claim. Since the flows of $V$ and $W$ are contained in the group $G$, we can assume, without loss of generality: by Lemma \ref{lem-generalpos} all the $x$-coordinates of the points $p_1, \dots, p_m, q$ are pair-wise different and non-zero, and all the $y$-coordinates of the points $p_1, \dots, p_m, q$ are pair-wise different and non-zero; moreover, by and Lemma \ref{lem-nonvanish}, $V$ and $W$ are spanning the tangent space of $Z_p$ in each of these points. 

Let $\gamma \colon [0,1] \to Z_p$ be a path in the Danielewski surface that connects $p_m$ to $q$, but avoids the points $p_1, \dots, p_{m-1}$ and, moreover, avoids the set $\{p'(z) = 0\} \cup \{x = 0\} \cup \{y = 0\}$, and is such that for every $\tau \in [0, 1]$ the point $\gamma(\tau)$ never has a common $x$-coordinate or a common $y$-coordinate with any of the points $p_1, \dots, p_{m-1}$. All these conditions remove only sets of complex codimension $1$ from the surface $Z_p$ and can therefore easily be satisfied. 

For every point $\gamma(\tau) \in Z_p, \tau \in [0,1]$, Lemma \ref{lem-interpolate} furnishes a polynomial $f(x) \in \CC[x]$ that vanishes in the $x$-coordinates of $p_1, \dots, p_{m-1}$ but not in the $x$-coordinate of $\gamma(\tau)$. Similarly, we obtain a polynomial $g(y) \in \CC[y]$ that vanishes in the $y$-coordinates of $p_1, \dots, p_{m-1}$ but not in the $y$-coordinate of $\gamma(\tau)$. According to Corollary \ref{cor-unipotentpowers} and Lemma \ref{lem-interpolate}, the degrees of $f$ and $g$ can be chosen such that $f W$ and $g W$ lie in the Lie algebra generated by $V, W, yV, xW$.

The map $\CC^2 \ni (t,s) \mapsto \psi_{f(x) s} \circ \varphi_{g(y) t}(\gamma(\tau)) \in Z_p$ is a submersion near $(0,0)$ that fixes $p_1, \dots, p_{m-1}$.   
Since $[0,1]$ is compact, we find a finite partition $0 = \tau_0 < \tau_1 < \dots < \tau_M = 1$ of $[0,1]$ such that
 for each closed interval $[\tau_j, \tau_{j+1}]$ of the partition and a point $\tau_j^\prime \in [\tau_j, \tau_{j+1}]$
 there exists an open neighborhood of $\gamma([\tau_j, \tau_{j+1}])$ where $\CC^2 \ni (t,s) \mapsto \alpha_j := \psi_{f(x) s} \circ \varphi_{g(y) t}(\gamma(\tau_j^\prime)) \in Z_p$ is submersive onto this neighborhood.
 
By Proposition \ref{prop-lieapprox} we can now approximate each $\alpha_j$ arbitrarily well on compacts by certain compositions $\widetilde{\alpha}_j$ of the flows of the generators $V, W, yV xW$ of the Lie algebra $\mathfrak{g}$. We can choose the approximation such that it is again submersive onto $\gamma([\tau_j, \tau_{j+1}])$. Thus, by the implicit function theorem we can choose times $(t,s)$ for each $j$ such that $\widetilde{\alpha}_j(\tau_j) = \widetilde{\alpha}_j(\tau_{j+1})$ while fixing $p_1, \dots, p_{m-1}$. 
Then, $\alpha = \widetilde{\alpha}_0 \circ \dots \circ \widetilde{\alpha}_{M-1} \in G$ is the desired automorphism.
\end{proof}

\begin{remark}
The smoothness of the Danielewski surface $Z_p$ is necessary for Theorem \ref{thm-unipotent}, but only needed in Lemma \ref{lem-nonvanish} and in Lemma \ref{lem-generalpos}. The other remaining calculations remain valid for a singular Danielewski surface. The same strategy of proof can then be applied to prove infinite transitivity on the smooth locus of a singular Danielewski surface.
\end{remark}

\section{Lie algebra of volume-preserving polynomial vector fields}

\label{sec-volume}

The smooth Danielewski surface $Z_p$ is equipped with a complex algebraic volume form 
\begin{equation}
\omega := \frac{dx \wedge dz}{x} = \frac{dz \wedge dy}{y} = \frac{dx \wedge dy}{p'(z)} \label{eq-volumeform}
\end{equation}

Since $Z_p$ is a surface, the volume form is in fact also a symplectic form, and we can use the formalism provided by Hamiltonian systems. 

It is elementary to find the following Hamiltonian functions for the vector fields $V, W, H$:
\begin{align*}
i_V \omega &= d y \\
i_W \omega &= d (-x) \\
i_{H} \omega &= d (-z)
\end{align*}

We can read off the Poisson bracket directly from Equation \eqref{eq-volumeform} and obtain:
\begin{align*}
\{x, z\} &= -x \\ 
\{z, y\} &= -y \\
\{x, y\} &= -p'(z)
\end{align*}
which corresponds to
\begin{align*}
[-W, -H] &= W \\ 
[-H, V] &= -V \\
[-W, V] &= p''(z) H \qedhere
\end{align*}

Using the identity $x y = p(z)$, every polynomial Hamiltonian function can be written uniquely (up to an additive constant) in the following form:
\begin{equation}
\label{eq-Hamiltonpossible}
H \in x\CC[x, z] \oplus y\CC[y, z] \oplus \CC[z]
\end{equation}

By the result of Kutzschebauch and Leuenberger \cite{MR3495426} we know that the shear vector fields of $V$ and $W$, and, in fact, even all locally nilpotent derivations together, do not generate the Lie algebra of volume-preserving polynomial vector fields on a smooth Danielewski surface. Hence, it is clear that our result will necessarily involve also the vector field $H$ and its shears. 

\begin{theorem}
\label{thm-volume}
The Lie algebra of volume-preserving holomorphic vector fields on a smooth Danielewski surface is generated by the following complete volume-preserving vector fields:
\begin{align*}
y^n V, &\quad n = 0, \dots, \max(1, d-3) \\
x^n W, &\quad n = 0, \dots, \max(1, d-3) \\
z^m H, &\quad m = 0, \dots, \max(2, d-3)
\end{align*}
\end{theorem}
\begin{proof}
The Hamiltonian functions corresponding to the given vector fields (up to multiplicative and additive constants) are the following:
\[
y^{n+1}, x^{n+1}, z^{n+1}
\]
By Corollary \ref{cor-unipotentpowers} we obtain the vector fields $y^n V$ and $x^n W$ for all $n \geq 0$. 
We compute the following Poisson brackets by induction on $m$:
\begin{align*}
\{ x^n, z^2 \} &= -2 n x^n z \\
\{ x^n z^m, z^2 \} &= -2 n x^n z^{m+1} \\
\{ y^n, z^2 \} &= 2 n y^n z \\
\{ y^n z^m, z^2 \} &= 2 n y^n z^{m+1}
\end{align*}
This yields the first two summands in Equation \eqref{eq-Hamiltonpossible}. We further observe that
\[
\left\{ x z^k, y z^m  \right\} = -\left(z^{k+m} \cdot  p(z)\right)'
\]
The missing terms from this computation are $z, z^2, \dots z^{d-2}$ which were included in the assumptions of the theorem.
\end{proof}

\section{Direct Products}
\label{sec-products}

In this section we collect two more general results on direct products.

\begin{proposition}
\label{prop-prodline}
Let $Z$ be a complex-affine manifold that admits finitely many LNDs generating an infinitely transitive action. Then $Z \times \CC$ also admits finitely many LNDs generating an infinitely transitive action.
\end{proposition}
\begin{proof}
We may assume that $Z \subset \CC^N$ is a smooth affine subvariety, where the dimension $N$ is minimal in the following, rather weak sense: Every projection of $Z$ to any hyperplane of $\CC^N$ is not injective. We denote the coordinate functions of $\CC^N$ by $z_1, \dots, z_N$. 
Let $\Theta_1, \dots, \Theta_m$ be the finitely many (non-vanishing) LNDs on $Z$ that generate an infinitely transitive action on $Z$. We denote their trivial extension to $Z \times \CC$ the same way. Let $G$ denote the subgroup of automorphisms of $Z \times \CC$ that is generated by the flows of $\Theta_1, \dots, \Theta_m$.
Let $w$ be the variable in $\CC$ and set $\Xi := z_1 \cdot \frac{\partial}{\partial w}$ and $\Omega := w \cdot \Theta_1$ which defines two LNDs on $Z \times \CC$. We claim that the flows of $\Theta_1, \dots, \Theta_m, \Xi, \Omega$ generate an infinitely transitive action on $Z \times \CC$: 

\begin{enumerate}
\item We choose a sequence without accumulation points and without repetition $(s_k')_{k \in \NN} \subset Z \cap (\CC^\ast \times \CC^{N-1})$ of ``standard points'' such that $\Theta_1$ does not vanish in any of them, and such that their $z_1$-coordinates are pairwise different. Note that the algebraic map $(z_1, \dots, z_N) \mapsto z_1$ from $Z$ to $\CC$ is surjective since it can't be constant, and a generic fibre will have co-dimension $1$ in $Z$, and that $\Theta_1$ vanishes at most on a subvariety of codimension $1$.
\item Let $r \in \NN$ and let $p_1 = (p_1',p_1''), \dots, p_r = (p_r',p_r'') \in Z \times \CC$ be given $r$ pairwise distinct points. We prove the claim if we find an automorphism using the flows of $\Theta_1, \dots, \Theta_m, \Xi, \Omega$ that maps $p_j$ to $(s_j'	,0)$ for all $j = 1, \dots, r$.
\item Let $\pi' \colon Z \times \CC \to Z$ and $\pi'' \colon Z \times \CC \to \CC$ be the projections to the respective factors. 
There exists $g \in G$ such that $\pi'(g(p_j)) \subset \{ s_k' \}_{k \in \NN}$ for all $j = 1, \dots, r$. 
\item For $\pi'(g(p_j)) = \pi'(g(p_k))$ and $j \neq k$ we necessarily have that $\pi''(g(p_j)) \neq \pi''(g(p_k))$. Hence, the flow map $F_t$ of $\Omega = w \cdot \Theta_1$ will change their $\pi'$-projections, and for a sufficiently small flow-time $t$ we can ensure that $\pi'(F_t \circ g(p_j)) \neq \pi'(F_t \circ g(p_k))$, and -- in fact -- even pairwise different $z_1$-coordinates for $j \neq k$.
\item Without loss of generality we may assume that $\pi''(F_t \circ g(p_j)) \neq 0$ for $j = 1, \dots, r$. Otherwise, apply the flow map of $\Xi$ for an arbitrarily small time to ensure this holds. Its application will not change the $\pi'$-projections. 
\item There exists $h \in G$ such that the coordinate function $z_1$ evaluated in $\pi'(h \circ F_t \circ g(p_j))$ equals $-\pi''(h \circ F_t \circ g(p_j))$ for all $j = 1, \dots, r$. Note that the $z_1$-coordinates are pairwise different.
\item Applying first the time-$1$ map of $\Xi$ sends $h \circ F_t \circ g(p_j)$ into $Z \times \{0\}$. Then, applying another element from $G$ sends these points to the standard points $(s_j',0)$ for all $j = 1, \dots, r$. \qedhere
\end{enumerate}
\end{proof}

\begin{proposition}
\label{prop-directprod}
Let $X$ and $Y$ be complex-affine manifolds that each admit finitely many locally nilpotent derivations generating an infinitely transitive action on $X$ and $Y$, respectively. 
Then $X \times Y$ admits finitely many locally nilpotent derivations generating an infinitely transitive action on the direct product $X \times Y$.
\end{proposition}
\begin{proof}
Let $\Theta_1, \dots, \Theta_m$ be the finitely many (non-vanishing) LNDs on $X$ that generate an infinitely transitive action on $X$. And let $\Xi_1, \dots, \Xi_n$ be the finitely many (non-vanishing) LNDs on $Y$ that generate an infinitely transitive action on $Y$. We denote the trivial extensions of all these LNDs to the direct product $X \times Y$ by the same symbols.

By $G_X$ and $G_Y$ we denote the subgroups of the automorphism group of $X \times Y$ that are generated by the flows of $\Theta_1, \dots, \Theta_m$ and the flows of $\Xi_1, \dots, \Xi_n$, respectively.

Let $X \subset \CC^M$ and $Y \subset \CC^N$ be embedded as smooth affine subvarieties. We denote the coordinates in $\CC^M$ and $\CC^N$ by $(z_1, \dots, z_M)$ and $(w_1, \dots, w_N)$, respectively.

We choose a sequence without accumulation points and without repetition $(s_k')_{k \in \NN} \subset X \cap (\CC^\ast)^M$ of ``standard points'' such that $\Theta_1$ does not vanish in any of them. 
And we choose a sequence without accumulation points and without repetition $(s_k'')_{k \in \NN} \subset Y  \cap (\CC^\ast)^N$ of ``standard points'' such that $\Xi_1$ does not vanish in any of them. 

Let $r \in \NN$ and let $p_1 = (p_1',p_1''), \dots, p_r = (p_r',p_r'') \in X \times Y$ be given $r$ pairwise distinct points.

By an action of $G_X$ and of $G_Y$ we may assume that $p_j' \in \{s_k'\}_{k \in \NN}$ and  $p_j'' \in \{s_k''\}_{k \in \NN}$ for all $j = 1, \dots, r$.

We can choose small flow times of $z_1 \cdots z_N \cdot \Xi_1$ and $w_1 \cdots w_N \cdot \Theta_1$ to ensure that all coordinates $p_1', \dots, p_r'$ are pairwise different, and also all coordinates $p_1'', \dots, p_r''$ are pairwise different.

By an action of $G_X$ and of $G_Y$ we may assume that $p_j' = s_j'$ and  $p_j'' = s_j''$ for all $j = 1, \dots, r$.
\end{proof}

\section{Affine modification of $\CC^n$}

\label{sec-suspension}

Let $Z \subset \CC^N$ be an affine algebraic variety and let $p \in \CC[Z]$ be a polynomial. We consider the \emph{affine modification}
\[
Z_p := \{ (x,y,z) \in \CC \times \CC \times Z \,:\, x \cdot y - p(z) = 0 \}
\]

For $k = 1, \dots, N$ we define the following vector fields on $Z_p$:
\[
\begin{alignedat}{4}
V_k &= \frac{\partial p}{\partial z_k} \cdot \frac{\partial}{\partial x} & &+ y \cdot \frac{\partial}{\partial z_k} \\
W_k &=  & \frac{\partial p}{\partial z_k} \cdot \frac{\partial}{\partial y} &+ x \cdot \frac{\partial}{\partial z_k} \\
  H &= -x \cdot \frac{\partial}{\partial x} &+ y \cdot \frac{\partial}{\partial y}
\end{alignedat}
\]
The vector fields $V_k, W_k$ are easily seen to be locally nilpotent derivations, and $H$ induces a $\CC^\ast$-action. They were already studied in the same context by Kaliman and Kutzschebauch  \cite{MR2350038}*{Lemma 2.6}. Moreover, we have that
\begin{align*}
y, z_\ell \in &\ker V_k, \qquad \ell \neq k \\
x, z_\ell \in &\ker W_k, \qquad \ell \neq k 
\end{align*}

For the commutators, we obtain:
\begin{align*}
[V_k, W_k]    &= \frac{\partial^2 p}{\partial z_k^2} \cdot H \\
[V_k, W_\ell] &= \frac{\partial^2 p}{\partial z_k \partial z_\ell} \cdot H + \frac{\partial p}{\partial z_k} \cdot \frac{\partial}{\partial z_\ell} - \frac{\partial p}{\partial z_\ell} \cdot \frac{\partial}{\partial z_k} \\
[V_k, V_\ell] &= 0 \\
[W_k, W_\ell] &= 0 \\
[V_k, H] &= -V_k\\
[W_k, H] &= W_k
\end{align*}

\begin{lemma}
\label{lem-severaldimgen}
Assume that $\frac{\partial p }{\partial z_k} \neq 0$ for every $k \in \{ 1, \dots, N \}$.
Then the Lie algebra generated by $V_k, W_k, yV_k, xW_k$ with $k \in \{ 1, \dots, N \}$ contains the vector fields
\[
y^{n} V_k, x^{n} W_k \qquad n \geq \max(0, d - 2)
\]
where $d$ is the total degree of $p$.
\end{lemma}

\begin{proof}
We first apply Lemma \ref{lem1} for the variable $z_k$ instead of $z$. Since $V_k$ and $W_k$ do not touch the other variables $z_\ell$ for $\ell \neq k$, we obtain
\begin{align*}
\adj_{V_k^n}(W_k)        &=  -(n-1) y^{n-2} \frac{\partial^n p}{\partial z_k^n} \cdot V_k 
\end{align*}
for $n$  here being the degree of $p$ in $z_k$.

Next, we subsequently take the adjoint actions of $V_\ell$ on this result for every variable $z_\ell$ that appears, as many times as the degree of the polynomial in $z_\ell$. Note that
\[
[V_\ell, z_\ell^m y^n V_k] = y^n V_\ell(z_\ell^m) V_k = m y^{n+1} z_\ell^{m-1} V_k
\]
Hence, we obtain $y^M V_k$ for some $M \leq d-2$. Since the leading term w.r.t.\ $z_k$ might not be the leading term for the total degree, it is possible that $M < d-2$.

We proceed as in Section \ref{sec-LND}: Lemma \ref{lem3} gives
\begin{align*}
\adj^m_{V_k}(\adj^n_{yV_k}(W_k)) &= -(n+m) \cdot y^{2n-2+m} \cdot \frac{\partial^{n+m} p}{\partial z_k^{n+m}} \cdot V_k
\end{align*}
for any choice of $n$ and $m$ such that $n+m$ equals the degree of $p$ in $z_k$. 
It is important to note that for this choice, we will end up with exactly the same polynomial in $z_\ell, \ell \neq k,$ as in the previous step. Hence, by exactly the same adjoint actions of $V_\ell$ as in the previous step, namely in total $M-(n+m)+2$, we obtain $y^{M + n} V_k$ for the same $M \leq d-2$ and with $0 \leq n \leq \deg_{z_k} p$.
By the assumption $\frac{\partial p}{\partial z_k} \neq 0$ we ensure that $\deg_{z_k} p \geq 1$.

Similarly, we proceed with the applications of Lemma \ref{lem4} and Corollary \ref{cor-unipotentpowers} to finally obtain $y^M V_k, y^{M+1} V_k, y^{M+2} V_k, \dots$. Analogously, we obtain the powers of $x$ in front of $W_k$. 
\end{proof}

\begin{lemma}
Let $Z = \CC^N$ and $\{ dp = 0 \} \cap \{ p=0 \} = \emptyset$. The vector fields $\{ V_k, W_k \,:\, k = 1, \dots, N \}$ span the tangent space of every point $(x,y,z) \in Z_p$ that satisfies
\begin{enumerate}
\item $x \neq 0$ and $\exists k \in \{1, \dots, N\}$ s.t.\ $\frac{\partial p}{z_k}(z) \neq 0$, or
\item $y \neq 0$ and $\exists k \in \{1, \dots, N\}$ s.t.\ $\frac{\partial p}{z_k}(z) \neq 0$
\end{enumerate}
\end{lemma}
\begin{proof}
For $y \neq 0$, the vector fields $\{ V_1, \dots, V_N \}$ are linearly independent and span an $N$-dimensional subspace of the tangent space of $Z_p$. For any point $(x,y,z) \in Z_p$ where $\frac{\partial p}{z_k}(z)$ does not vanish, $W_k$ will be linearly independent from $\{ V_1, \dots, V_N \}$, and together they span the tangent space of $Z_p$. Similarly, we can argue with the roles of $V_k$ and $W_k$ reversed. 
\end{proof}

In the following, we will denote by $G_p$ the group generated by the flows of the locally nilpotent derivations $\{ V_k, W_k \,:\, k = 1, \dots, N \}$ on $Z_p$.

\begin{lemma}
\label{lem-affmodgeneralpos}
Let $p_1, \dots, p_r \in Z_p$ and $\{ dp = 0 \} \cap \{ p=0 \} = \emptyset$. Then there exists $g \in G_p$ such that for all points we have $g(p_1) \notin A, \dots, g(p_r) \notin A$, where $A:= \{ x = 0\} \cup \{ y = 0\} \cup \{ dp(z) = 0\}$. Moreover, $Z_p \setminus A$ is path-connected.
\end{lemma}
\begin{proof}
The set $A = \{ x = 0\} \cup \{ y = 0\} \cup \{ dp(z) = 0\}$ is of complex codimension $1$ in $Z_p$ since $\{ dp = 0 \} \cap \{ p=0 \} = \emptyset$. Hence, $Z_p \setminus A$ is path-connected.
If $a \in A$ and $x = 0 \vee y = 0$, which implies $p(z) = 0$, then $dp(z) \neq 0$. Hence, there exists $k \in \{1, \dots, N\}$ such that $V_{k,a} \neq 0$ and $W_{k,a} \neq 0$. Flowing for an arbitrarily small time along these two will move the point $a$ outside the set $A$ since $V_{k,a}$ has non-vanishing $\frac{\partial}{\partial x}$-component and $W_{k,a}$ has non-vanishing $\frac{\partial}{\partial y}$-component.
If $a \in A$ and $\{ dp(z) = 0\}$ but $x \neq 0, y \neq 0$ then flowing for an arbitrarily small time along any of the $V_k, W_k$ will move $a$ outside the set $A$ by changing the $z_k$-coordinate.
Since we are considering only finitely many points $p_1, \dots, p_r$, we therefore find an element $g \in G_p$ such that $g(p_1) \notin A, \dots, g(p_r) \notin A$.
\end{proof}

\begin{theorem}
\label{thm-affinemod}
Let $p \in \CC[z_1, \dots, z_N]$ be a polynomial of total degree $d \geq 1$ with smooth reduced zero fiber. Then the flows of the finitely many locally nilpotent derivations 
\[
\{ V_k, \; W_k, \; y V_k, \; x W_k \,:\, k = 1, \dots, N \}
\]
generate a group that acts infinitely transitively on
\[
Z_p := \{ (x,y,z) \in \CC \times \CC \times \CC^N \,:\, x y = p(z) \}
\]
\end{theorem}

\begin{proof}
After an affine-linear change of coordinates, we may assume without loss of generality that $p(0) = 1$ which will be used at the end of the proof.
If $\frac{\partial p }{\partial z_k} = 0$ for some $k$, then $Z_p$ is naturally isomorphic to a direct product $\widetilde{Z_p} \times \CC$ where $\widetilde{Z_p}$ is given by the polynomial $p$, but considered in $\CC[z_1, \dots, z_{k+1}, z_{k+1}, \dots, N]$ instead. By Proposition \ref{prop-prodline} we may therefore assume without loss of generality that $\frac{\partial p }{\partial z_k}$ does not vanish identically for any $k = 1, \dots, N$. We can therefore apply Lemma \ref{lem-severaldimgen} and obtain all vector fields of the form
\[
\{ y^n V_k, \; x^n W_k \,:\, k = 1, \dots, N, \; n \geq \max(0, d - 2) \}
\]
Let $p_1, \dots, p_r \in Z_p$ pairwise distinct points. By Lemma \ref{lem-affmodgeneralpos} we may further assume that $p_1, \dots, p_r \in Z_p \setminus A$ where $A = \{ x = 0\} \cup \{ y = 0\} \cup \{ dp(z) = 0\}$. 

Note that the vector fields $V_k$ and $W_k$ do not change any of the coordinates $z_\ell$ for $\ell \neq k$. For fixed $z_\ell, \ell \neq k,$ we therefore obtain a Danielewski surface, and we can apply Theorem \ref{thm-unipotent} to change all the $z_k$-coordinates of $g(p_1), \dots, g(p_r)$ to zero for each $k$. These points lie on the curve $x y = p(0) = 1, z_1 = \dots = z_N = 0$. We may now assume that this curve is contained in Danielewski surface for $N=1$ where we can move them using $y^n V_1$ and $x^n W_1$ to a standard set of points, say $(x,y,z_1) = (1,1,0), (2,1/2,0), (3, 1/3, 0)$ etc.
This can be obtained again by Theorem \ref{thm-unipotent}. 
\end{proof}

\section*{Funding}

The author was supported by the European Union (ERC Advanced grant HPDR, 101053085 to Franc Forstneri\v{c}) and grant N1-0237 from ARRS, Republic of Slovenia.

\end{document}